\documentclass[11pt,oneside,english,hidelinks]{amsart}
\usepackage[T1]{fontenc}
\usepackage[utf8]{inputenc}
\usepackage[dvipsnames]{xcolor}
\usepackage[yyyymmdd,hhmmss]{datetime}
\usepackage[numbers,sort,compress]{natbib}
\usepackage{amstext,amsthm,amssymb}
\usepackage{xparse,mathrsfs,mathtools}
\usepackage[english]{babel}
\usepackage[pdftex,unicode=true,%
	pdfusetitle,pdfdisplaydoctitle=true,%
	pdfpagemode=UseOutlines,%
	bookmarks=true,bookmarksnumbered=false,bookmarksopen=true,bookmarksopenlevel=2,%
	breaklinks=true,%
	pdfencoding=unicode,psdextra,
	pdfcreator={},
	pdfborder={0 0 0}
]{hyperref}
\usepackage{lmodern,microtype}
\usepackage[a4paper]{geometry}
\usepackage[capitalise,nameinlink]{cleveref}

\newtheorem*{thm*}{Theorem}
\newtheorem{thm}{Theorem}[section]
\newtheorem{lem}[thm]{Lemma}
\newtheorem{cor}[thm]{Corollary}
\theoremstyle{remark}

\newtheorem{question}[thm]{Question}
\newtheorem*{rem*}{Remark}

\newcommand{\NN}{\mathbb{N}}
\newcommand{\ZZ}{\mathbb{Z}}
\newcommand{\QQ}{\mathbb{Q}}
\newcommand{\RR}{\mathbb{R}}
\newcommand{\CC}{\mathbb{C}}
\newcommand{\PP}{\mathbb{P}}
\newcommand{\B}{\mathcal{B}}
\newcommand{\Prob}{\mathbb{P}}
\newcommand{\Expectation}{\mathbb{E}}
\DeclareMathOperator{\Var}{Var}
\newcommand{\dd}[1]{\mathop{\mathrm{d}#1}}
\NewDocumentCommand\e{ s O{} m }{%
	\IfBooleanTF{#1}{%
		\operatorname{e}_{#2}\parentheses[\big]{#3}%
	}{\operatorname{e}_{#2}\parentheses{#3}}%
}
\newcommand{\meas}{\operatorname{meas}}
\newcommand{\id}{\operatorname{id}}
\newcommand{\LandauO}{O}
\newcommand{\exceptionalF}{f_0}
\newcommand{\indic}{\boldsymbol{1}}
\DeclarePairedDelimiter\parentheses{\lparen}{\rparen}
\DeclarePairedDelimiter\braces{\lbrace}{\rbrace}
\DeclarePairedDelimiter\brackets{\lbrack}{\rbrack}
\DeclarePairedDelimiter\abs{\lvert}{\rvert}
\DeclarePairedDelimiter\norm{\lVert}{\rVert}
\DeclarePairedDelimiter\floor{\lfloor}{\rfloor}
\DeclarePairedDelimiter\lopeninterval{\lparen}{\rbrack}
\NewDocumentCommand\set{ s o m }{\IfBooleanTF{#1}{\braces*{\,#3\,}}{\IfNoValueTF{#2}{\braces{\,#3\,}}{\braces[#2]{\,#3\,}}}}

\crefformat{equation}{#2(#1)#3}
\crefformat{enumi}{#2(#1)#3}
\numberwithin{equation}{section}
\usepackage{xifthen,ifthen}
\makeatletter
\newcounter{@ToDo}
\newcommand{\todo@helper}[1]{%
	({\color{blue}TODO~\arabic{@ToDo}: {#1\@addpunct{.}}})%
}
\newcommand{\todo}[1]{\stepcounter{@ToDo}%
	\relax\ifmmode\text{\todo@helper{#1}}%
	\else\todo@helper{#1}\fi%
}
\makeatother

\title{Metric results on summatory arithmetic functions on Beatty sets}
\date{\today{}}


\author{Marc~Technau}
\address{
	Marc~Technau\\%
	Institut für Analysis und Zahlentheorie\\%
	TU~Graz\\%
	Kopernikusgasse~24/II\\%
	8010~Graz\\%
	Austria}
\email{mtechnau@math.tugraz.at}

\author{Agamemnon~Zafeiropoulos*}
\thanks{*Research supported by Austrian Science Fund~(FWF) projects~{F-5512} and~{Y-901}.} 
\address{
	Agamemnon~Zafeiropoulos\\%
	Department of Mathematical Sciences\\%
	Norwegian University of Science and Technology\\%
	NO-7491~Trondheim\\%
	Norway}
\email{agamemnon.zafeiropoulos@ntnu.no}

\begin{document}
\begin{abstract}
	Let $f\colon \NN\rightarrow \CC$ be an arithmetic function and consider the Beatty set $\B(\alpha) = \set{ \floor{n\alpha} : n\in\NN }$ associated to a real number $\alpha$, where $\floor{\xi}$ denotes the integer part of a real number $\xi$.
	We show that the asymptotic formula
	\[
		\abs[\Big]{
			\sum_{\substack{ 1\leq m\leq x \\ m\in \B(\alpha) }} f(m)
			- \frac{1}{\alpha} \sum_{1\leq m\leq x} f(m)
		}^2
		\ll_{f,\alpha,\varepsilon}
		(\log x)
		(\log\log x)^{3+\varepsilon}
		\sum_{1\leq m\leq x} \abs{f(m)}^2
	\]
	holds for almost all $\alpha>1$ with respect to the Lebesgue measure.
	This significantly improves an earlier result due to Abercrombie, Banks, and Shparlinski.
	The proof uses a recent Fourier-analytic result of Lewko and {Radziwi\l\l} based on the classical Carleson--Hunt inequality.\\
	Moreover, using a probabilistic argument, we establish the existence of functions $f\colon\NN\to\set{\pm 1}$ for which the above error term is optimal up to logarithmic factors.
\end{abstract}
\maketitle

\section{Introduction}

For a real number $\alpha\geq 1$ consider the associated \emph{Beatty set}\footnote{%
	The term \emph{Beatty sequence} is also commonly used, but we choose to refer to them as sets here, because we do not take advantage of their order-theoretic properties.%
} $\B(\alpha) = \set{ \floor{n\alpha} : n\in\NN }$, where $\floor{\xi} = \min\set{ m\in\ZZ : m\leq \xi }$ denotes the \emph{integer part} of a real number $\xi$.
Such sets may be viewed as a generalisation of arithmetic progressions.
Indeed, consecutive elements in the Beatty set $\B(\alpha)$ are spaced either $\floor{\alpha}$ or $\floor{\alpha}+1$ apart, and arithmetic progressions $a$, $2a$, $3a$, \ldots{} (with positive integer~$a$) are included as a special case.

In this article we are concerned with the following striking metric result on summatory functions on Beatty sets which was obtained by Abercrombie, Banks, and Shparlinski~\cite{abercrombie2009arithm} in~2008:
\begin{thm}[Abercrombie--Banks--Shparlinski]
	\label{thm:AbercrombieBanksShparlinski}
	Fix $\varepsilon>0$.
	Let $f\colon\NN\to\CC$ be any arithmetic function, and for $x\geq 1$ put
	\begin{gather}
		\label{eq:Mfx:Def}
		M(f,x) = 1 + \!\! \max_{1\leq m\leq x} \abs{f(m)}, \\
		\label{eq:Sfx:Def}
		S_{\alpha}(f,x) = \sum_{\substack{ 1\leq m\leq x \\ m\in \B(\alpha) }} f(m),
		\quad
		S(f,x) = \sum_{1\leq m\leq x} f(m)
	\end{gather}
	Then, for all $\alpha>1$ in a set depending on~$\varepsilon$ and~$f$, having full Lebesgue measure, one has
	\begin{equation}\label{eq:AsymptFormula}
		S_{\alpha}(f,x)
		= \alpha^{-1} S(f,x)
		+ \LandauO_{f,\alpha,\varepsilon}\parentheses*{ x^{\frac{3}{4}+\varepsilon} M(f,x) },
	\end{equation}
	where the implicit constant does not depend on~$x$.
\end{thm}
Here and throughout we use the Landau notation $g(x) = \LandauO(h(x))$ and the Vinogradov notation $g(x) \ll h(x)$ to mean that there exists some constant $C>0$ such that $\abs{g(x)} \leq C h(x)$ for all possible values of~$x$. Whenever the implicit constant $C>0$ depends on some parameters $\delta_1,\ldots,\delta_k$ we indicate this via subscripts, e.g., $g(x)=\LandauO_{\delta_1,\ldots,\delta_k}(h(x))$.
We write $g(x) \sim h(x)$ as $x\to \infty$ to mean\footnote{When there are more parameters involved, convergence is not necessarily uniform with respect to these parameters. } $\lim_{x\to\infty} g(x) / h(x) = 1$, while we write $g(x) \nsim h(x)$ to mean that $g(x) \sim h(x)$ does not hold.

\medskip

We pause for a moment to discuss why \cref{thm:AbercrombieBanksShparlinski} may be surprising.
Consider, for instance, the characteristic function $\indic_\PP$ of the primes.
Then, by the prime number theorem, $S(\indic_\PP,x) \sim x/\log x$ as $x\to\infty$. However, trivially, the Beatty set $\B(a)$ with integral~$a>1$ contains at most one prime, so that $S_a(\indic_\PP,x) = \LandauO(1)$, which shows a huge disparity between $S_a(\indic_\PP,x)$ and $a^{-1} S(\indic_\PP,x)$.\\
For rational $a/q$ ($a,q$ coprime), it is easy to see that $\B(\alpha)$ is the union of arithmetic progressions $\floor{na/q} + a \NN_0$, $1\leq n\leq q$.
Hence, by the prime number theorem for arithmetic progressions, the quest for evaluating $S_{a/q}(\indic_\PP,x)$ asymptotically pertains to how often $\floor{na/q}$ and $a$ happen to be coprime as $n$ ranges from~$1$ to~$q$.\\
For irrational $\alpha>1$ the situation changes completely; here one always has $S_{\alpha}(\indic_\PP,x) \sim \alpha^{-1} S(\indic_\PP,x)$ as $x\to\infty$ by results of Vinogradov \cite[Notes to Chapter~{XI}]{vinogradov2004themethod} (see also \cite[Chapter~{4.V}]{ribenboim1996new-book} and \cite{banks2009primenumbers}).

Motivated by the above, Abercrombie~\cite{abercrombie1995beatty-sequences} was lead to study the situation when $f$ equals the \emph{divisor function} $\tau = \tau_2$, where $\tau_k$ counts the number of representations of its argument as a product of $k$ positive integers.
He was able to show that, for every irrational~$\alpha>1$, $S_{\alpha}(\tau,x) \sim \alpha^{-1} S(\tau,x)$ as $x\to\infty$. Moreover, by constructing an uncountable set of counter-examples, he could show that one cannot in general improve the error term in the asymptotic formula $S_{\alpha}(\tau,x) = \alpha^{-1} S(\tau,x) + o_{\alpha}(S(\tau,x))$ as $x\to\infty$ (see \cref{thm:Abercrombie} below for the precise statement).\\
However, for almost all $\alpha>1$ with respect to the Lebesgue measure, he showed that $S_{\alpha}(\tau,x) = \alpha^{-1} S(\tau,x) + O_{\alpha,\varepsilon}(x^{5/7+\varepsilon})$.
This was later improved and generalised to higher divisor functions $\tau_k$ by Zhai~\cite{zhai1997a-note-on-a}, and L\"u and Zhai~\cite{lu2004-the-divisor-problem}. They have, for almost all $\alpha>1$,
\begin{equation}\label{eq:BeattyDivisorProblem}
	S_{\alpha}(\tau_k,x) = \alpha^{-1} S(\tau_k,x) +
	\begin{cases}
		O_{\alpha,\varepsilon}(x^{(k-1)/k+\varepsilon}) & \text{if } 2\leq k\leq 4, \\
		O_{k,\alpha,\varepsilon}(x^{4/5+\varepsilon}) & \text{if } k\geq 5,
	\end{cases}
\end{equation}
which, in turn, is improved upon by \cref{thm:AbercrombieBanksShparlinski} for $k\geq 4$.

We further remark that the literature contains a wealth of results showing that $S_{\alpha}(f,x) \sim \alpha^{-1} S(f,x)$ as $x\to\infty$ (possibly restricted to certain subsequences of the integers) for various functions $f$ and irrational $\alpha$.
The interested reader may wish to confer~\citep{banks2008density-of-non-residues, banks2006nonresidues, banks2006short-character-sums, banks2007prime-divisors, banks2009primenumbers, begunts2004on-an-analogue, begunts2005on-the-distribution, guloglu2008sums-of-multiplicative, lao2009oscillations}.

In this regard it may also be worth pointing out that \cref{thm:AbercrombieBanksShparlinski} yields good results for arithmetic functions like the divisor function, or Euler's totient, but may fail to provide useful information for certain other applications.
The reader may imagine wanting to study $S_{\alpha}(f,x)$, when $f=\chi$ is a Dirichlet character modulo~$q$ or something similar (see, for instance,~\cite{banks2006nonresidues,banks2006short-character-sums,banks2008density-of-non-residues}, or~\cite{technau2018modular-hyperbolas} for a similar situation).
Here the main focus usually lies in beating the trivial bound $\abs{S_{\alpha}(\chi,x)} \leq x$.
For such problems, the crux with formulas like~\cref{eq:AsymptFormula} lies in the fact that one often needs them to hold when $x$ is of comparable size with $q$ and, moreover, one would like to have some sort of uniformity when switching from one $\chi$ to another (possibly with different conductor).
We do not resolve this particular problem here, and it seems unlikely that this can be done in such a generality in which \cref{thm:AbercrombieBanksShparlinski} applies.

\section{Main results}

Returning to the above discussion of $S_{\alpha}(\tau_k,x)$, it seems natural to ask whether the error term in \cref{thm:AbercrombieBanksShparlinski} can be improved.
For an arithmetic function $f\colon\NN\to\CC$ we write $\norm{f\rvert_{x}}_2$ for the $\ell^2$ norm of $f$ restricted to $\set{ 1,\ldots,\floor{x} }$, i.e.,
\[
	\norm{f\rvert_{x}}_2^2 = \sum_{1\leq m\leq x} \abs{f(m)}^2.
\]
Our main result is as follows.
\begin{thm}\label{thm:OurVersion}
	Fix $\varepsilon>0$.
	Let $f\colon\NN\to\CC$ be any arithmetic function.
	Keep the notation from~\cref{eq:Mfx:Def} and~\cref{eq:Sfx:Def}, and assume that $f$ does not grow too quickly in the sense that
	\begin{equation}\label{eq:GrowthBound}
		\limsup_{r\to\infty} \frac{ \norm{f\rvert_{2^{r+1}}}_2 }{ \norm{f\rvert_{2^r}}_2 } < \infty.
	\end{equation}
	Then, for all $\alpha>1$ in a set of full Lebesgue measure (depending on $f$), one has
	\begin{equation}\label{mainequation}
		S_{\alpha}(f,x)
		= \alpha^{-1} S(f,x)
		+ \LandauO_{f,\alpha,\varepsilon}\parentheses*{
			\norm{f\rvert_{x}}_2
			(\log x)^{\frac{1}{2}}
			(\log\log x)^{\frac{3}{2}+\varepsilon}
		},
	\end{equation}
	uniformly for $x\geq 8$.\footnote{The constant~$8$ has no particular importance; it is just there to ensure $x$ is large enough for the double logarithm to make sense.}
\end{thm}
We note that in applications the function $x\mapsto\norm{f\rvert_{x}}_2$ is generally well-behaved in the sense that it is $\asymp x^d (\log x)^k$ for some $d>0$ and $k\in\RR$ as $x\to\infty$. In such cases \cref{eq:GrowthBound} is immediately seen to be satisfied, so the condition~\cref{eq:GrowthBound} is quite mild.

For easier comparison with \cref{thm:AbercrombieBanksShparlinski}, we note that $\norm{f\rvert_{x}}_2 \leq \sqrt{x}\, M(f,x)$ where $M(f,x)$ is given by~\cref{eq:Mfx:Def}. Thus, \cref{thm:OurVersion} immediately implies the following result:
\begin{cor}
	Fix $\varepsilon>0$.
	Let $f\colon\NN\to\CC$ be any arithmetic function satisfying~\cref{eq:GrowthBound}.
	Keep the notation from~\cref{eq:Mfx:Def} and~\cref{eq:Sfx:Def}.
	Then, for all $\alpha>1$ in a set of full Lebesgue measure (depending on $f$), one has
	\[
		S_{\alpha}(f,x)
		= \alpha^{-1} S(f,x)
		+ \LandauO_{f,\alpha,\varepsilon}\parentheses*{
			x^{\frac{1}{2}+\varepsilon}
			M(f,x)
		},
	\]
	uniformly for $x\geq 8$.
\end{cor}

In view of the well-known formula
\[
	\norm{\tau_k\rvert_{x}}_2^2
	= \sum_{1\leq m\leq x} \tau_k(m)^2
	\sim C_k x (\log x)^{k^2-1}
	\quad\text{as }x\to\infty
\]
with some constant $C_k$ (see, e.g., \cite[§~II.13]{sandor2006handbook}), \cref{thm:OurVersion} improves~\cref{eq:BeattyDivisorProblem} for all $k\geq 2$.

We can also prove a version of \cref{thm:OurVersion} for so-called \emph{inhomogeneous Beatty sets}, stated as \cref{thm:OurVersion:Inhomogeneous} below.
Unfortunately, we are unable to handle this case in the full generality one might hope for. We discuss this in due time in \cref{sec:Comments} below.

\medskip

Returning to our goal of improving \cref{thm:AbercrombieBanksShparlinski}, we note that for specific choices of $f$ the error term in~\cref{mainequation} may be significantly smaller. Indeed, take for instance $f$ to be the constant zero function~$\boldsymbol{0}$. Then, for any $x,\alpha>1$, we have $S_\alpha(\boldsymbol{0},x) = \alpha^{-1} S(\boldsymbol{0},x) = 0$ with no error term at all. Moreover, taking $f$ to be the identity function~$\id_{\NN}\colon\NN\to\NN$, it is easy to see that
\[
	S_\alpha(\id_{\NN},x) = \alpha^{-1} S(\id_{\NN},x) + \LandauO_{\alpha}\parentheses*{ \norm{\id_{\NN}\rvert_{x}}_2 \, x^{-\frac{1}{2}} },
\]
and the error term therein is $\gg_\alpha \norm{\id_{\NN}\rvert_{x}}_2 \, x^{-\frac{1}{2}}$ infinitely often.

Nonetheless, the error term~\cref{mainequation} in our \cref{thm:OurVersion} is optimal up to logarithmic factors in the following sense:
\begin{thm}
	\label{thm:Optimality}
	There exists an arithmetic function $\exceptionalF\colon \NN\rightarrow\set{ \pm1 }$  such that, for almost all $\alpha>1$, the inequality
	\begin{equation} \label{lowerbound}
		S_{\alpha}(\exceptionalF,x) - \alpha^{-1} S(\exceptionalF,x)
		\gg_{\alpha}  \norm{\exceptionalF\rvert_x}_2 \sqrt{\log\log x}
	\end{equation}
	holds for infinitely many $x\in\NN$.
\end{thm}
The `badly behaved' function $\exceptionalF$ in \cref{thm:Optimality} is constructed using a probabilistic argument.
In particular, our proof shows that such functions are quite abundant in a suitable measure-theoretic sense, see \cref{sec:ProofOfOptimality} below.

To the best of our knowledge, \cref{thm:Optimality} is essentially the first result of its kind in the literature on Beatty sets.
(Albeit, of course, not at all novel in number theory as a whole; see, for instance, \cite[Chapter~{III}]{halberstam1966sequences} for similar metric results.)
The result which perhaps comes closest to ours is the following theorem due to Abercrombie~\cite[Theorem~{III}]{abercrombie1995beatty-sequences}, but his set of `bad' $\alpha$ is only shown to be large in a set-theoretic, but not measure-theoretic sense.
\begin{thm}[Abercrombie]\label{thm:Abercrombie}
	Let $g\colon\RR\to\RR$ be positive, increasing and unbounded.
	Then for uncountably many numbers $\alpha>1$ there exist arbitrarily large positive $x$ such that the relation
	\[
	\abs{ S_{\alpha}(\tau,x) - \alpha^{-1}S(\tau,x) } \leq S(\tau,x) / g(x)
	\]
	does not hold.
\end{thm}
We remark that upon comparing \cref{thm:Abercrombie} with our above discussion of the functions~$\boldsymbol{0}$ and~$\id_{\NN}$, it would be interesting to prove variants of \cref{thm:Abercrombie} for other `arithmetically interesting' functions different from $\tau$.
Likewise, it would be desirable to be able to replace $f_0$ in \cref{thm:Optimality} with some more concrete arithmetic function.
The latter seems to be more difficult.
We leave both problems as challenges to the reader.

\section{Where the improvement comes from}

Before proceeding to the proof of \cref{thm:OurVersion}, we shall give a short informal overview of the proof of \cref{thm:AbercrombieBanksShparlinski} in~\cite{abercrombie2009arithm}, and describe where our improvement comes from.
It turns out that one can detect membership to a Beatty set using Fourier analysis (see \cref{lem:BeattyMembership} and \cref{lem:SawtoothFourier} below).\\
For technical reasons, \cite{abercrombie2009arithm} are (essentially) working with truncations of the corresponding Fourier series.
This introduces another technical obstacle of having to deal with the fact that such truncations are not precisely equal to the function one `wants' to use, but this can be overcome and does not set the limit of the final error term (see~\cite[p.~87 up to bound~(14)]{abercrombie2009arithm}).

The main problem is then bounding a certain trigonometric polynomial, namely
\begin{equation}\label{eq:Qfx:def}
	Q_{\alpha}(f,x) = \sum_{1\leq\abs{k}\leq (x+1)\sqrt{x}} \parentheses[\Big]{ \sum_{\substack{ m\leq x+1 \\ \abs{j} \leq \sqrt{x} \\ mj = k }} g_{f,x}(m) c_x(j) } \exp(2\pi i k / \alpha)
\end{equation}
where the coefficients $c_x(j)$ satisfy
\begin{equation}\label{eq:cxj:def}
	c_x(j) \ll 1/\abs{j},
\end{equation}
and $g_{f,x}(m)$ is given by
\[
	g_{f,x}(m) = \begin{cases}
		\phantom{+{}} f(1) & \text{if } m = 1, \\
		- f(m-1) + f(m) & \text{if } 2\leq m \leq x, \\
		- f(m-1) & \text{if } x < m \leq x+1.
	\end{cases}
\]
Abercrombie, Banks, and Shparlinski show that
\[
	{\int_0^1 \abs{ Q_{\lambda^{-1}}(f,x) }^2 \dd{\lambda}}
	\ll x (\log x)^3 M(f,x)^2,
\]
i.e., $Q_{\alpha}(f,x)$ is small in an $L^2$ sense.
They then use a nice elementary argument to pass from this $L^2$ bound to a similar bound for $Q_{\alpha}(f,x)$ for almost all $\alpha>1$.
However, it is this step which comes with a certain loss of efficiency (see \cite[first bound on page~87]{abercrombie2009arithm}).

This can be avoided by combining the celebrated \emph{Carleson--Hunt inequality}~\cite{hunt1968on-the-convergence} (stated below) with standard arguments based on Chebyshev's inequality and the Borel--Cantelli lemma (this is carried out in \cref{lem:SigmaBound} below).
\begin{thm*}[Carleson--Hunt]
	For any sequence $(c_{k})_{k\in\ZZ}$ of complex numbers, and any positive integer~$Y$,
	\[
		{\int_{0}^{1} \parentheses[\bigg]{ \max_{1\leq y\leq Y} \abs[\Big]{\sum_{1\leq \abs{k}\leq y}c_k \exp(2\pi i k\lambda)} }^2 \!\dd{\lambda}}
		\ll \sum_{1\leq \abs{k}\leq Y} \abs{c_k}^2.
	\]
\end{thm*}
Here the fact that the inner summation in~\cref{eq:Qfx:def} depends on $x$ presents a technical difficulty.
This could be overcome by using a splitting-type argument as in~\cite[p.~86]{abercrombie2009arithm} and using the decay condition~\cref{eq:cxj:def} on the `tail.'
With more refined arguments one can do even better (see the work of Aistleitner, Berkes and Seip~\cite{aistleitner2015gcdsums}).
However, instead of using the results from~\cite{aistleitner2015gcdsums}, we shall employ a recent improvement due to Lewko and {Radziwi\l\l}~\cite{lewko2017refinements}.
Using their result, stated as \cref{thm:Lewko-Radziwill} below, allows for an appreciably short proof of \cref{thm:OurVersion}.
In fact, we can even afford to skip the whole truncation argument from above altogether.

\section{Proofs}

\subsection{Preliminaries}

Membership to $\B(\alpha)$ is characterised via the following well-known lemma:
\begin{lem}
	\label{lem:BeattyMembership}
	Let $\alpha \geq 1$ be a real number.
	Then, for any integer $m > \alpha-1$,
	\[
		m \in \B(\alpha)
		\;\Longleftrightarrow\;
		m/\alpha \in \lopeninterval*{ -1/\alpha, 0 } \bmod 1.
	\]
\end{lem}

It appears natural to treat the mod~$1$ condition in \cref{lem:BeattyMembership} using Fourier analysis.
However, one faces difficulties below when immediately working with the Fourier series of the characteristic function $\indic_{\alpha}$ of $\lopeninterval[\big]{ -1/\alpha, 0 } \bmod 1$, because the parameter $\alpha$ appears in the corresponding Fourier coefficients. (This would bar one from applying \cref{thm:Lewko-Radziwill} below.)
The following lemma is the key to circumventing this problem (see, e.g., \cite[p.~68]{iwaniec2004analytic}).
\begin{lem}
	\label{lem:SawtoothFourier}
	Let $\psi\colon\RR\to\RR$ denote the \emph{saw-tooth function} given by $\psi(t) = t - \floor{t} - \tfrac{1}{2}$. 
	Then the following assertions hold:
	\begin{enumerate}
		\item the Fourier coefficients $c_j$ of $\psi$ are $c_0 = 0$ and $c_j = - 1 / (2\pi i j)$ for $j\neq 0$;
		\item for real $\alpha>1$, the characteristic function $\indic_{\alpha}$ of $\lopeninterval[\big]{ -1/\alpha, 0 } \bmod 1$ satisfies
		\[
			\indic_{\alpha}(t) = \alpha^{-1} + \psi(t) - \psi(t+1/\alpha).
		\]
	\end{enumerate}
\end{lem}

\subsection{Proof of \texorpdfstring{\cref{thm:OurVersion}}{Theorem\autoref{thm:OurVersion}}}

By \cref{lem:BeattyMembership}
\[
	S_{\alpha}(f,x)
	= \sum_{\substack{ 1\leq m\leq x \\ \mathclap{ m > \alpha-1 } }} f(m) \indic_{\alpha}(m/\alpha),
\]
where $\indic_{\alpha}$ is as in \cref{lem:SawtoothFourier}.
Consequently,
\begin{equation}\label{eq:fBeattySummatory}
	S_{\alpha}(f,x)
	= \alpha^{-1} S(f,x)
	+ \varSigma_{\alpha}^{(0)}(f,x) - \varSigma_{\alpha}^{(1)}(f,x)
	+ \LandauO\parentheses{ \alpha M(f,\alpha) },
\end{equation}
where
\[
	\varSigma_{\alpha}^{(\ell)}(f,x)
	\coloneqq \sum_{1\leq m\leq x} f(m) \psi((m+\ell)/\alpha)
	\quad (\text{with } \ell \in \ZZ).
\]
In order to deal with $\varSigma_{\alpha}^{(\ell)}(f,x)$, we use the following theorem due to Lewko and {Radziwi\l\l}~\cite[Theorem~5]{lewko2017refinements}:
\begin{thm}[Lewko--{Radziwi\l\l}]
	\label{thm:Lewko-Radziwill}
	Let $\psi_*\colon\RR\to\CC$ be a complex-valued
	function with period one and Fourier coefficients $c_j$ satisfying\footnote{%
		In \cite[Theorem~5]{lewko2017refinements} the assumption that $c_0=0$ (i.e., ${\int_{0}^{1}\psi_{*}(t)\dd{t}} = 0$) is missing; however it is necessary in order for the theorem to be true.
	}
	$c_0=0$ and the decay condition $c_{j}\ll (1+\abs{j})^{-1}$. Let $(k_m)_{m\in\NN}$ be a strictly increasing sequence of positive integers and $(f_m)_{m\in\NN}$ a sequence of complex numbers.
	Then, for every $X\geq 8$,
	\[
		\int_0^1 \parentheses[\bigg]{
			\max_{1\leq x\leq X}\abs[\Big]{ \sum_{1\leq m\leq x} f_m \psi_*(k_{m}\lambda) }
		}^2 \!\dd{\lambda}
		\ll (\log\log X)^2 \sum_{1\leq m\leq X} \abs{f_m}^{2}.
	\]
\end{thm}

\noindent When applied to our present situation, this immediate yields the following result.
\begin{cor}
	\label{cor:SigmaL2Bound}
	Keeping the notation from above, we have
	\begin{equation}\label{eq:SigmaBound}
		\int_0^1 \parentheses[\bigg]{
			\max_{1\leq x\leq X}\abs[\big]{ \varSigma^{(\ell)}_{\lambda^{-1}}(f,x) }
		}^2 \!\dd{\lambda}
		\ll (\log\log X)^2 \norm{f\rvert_X}_2^2.
	\end{equation}
\end{cor}

As described above, we can pass from an $L^2$ bound as in \cref{cor:SigmaL2Bound} to a metric pointwise bound at essentially no extra cost---the bound is worsened only by something slightly larger than the square root of a logarithm:
\begin{lem}
	\label{lem:SigmaBound}
	Let $\ell$ be an arbitrary integer and keep the notation from above.
	Furthermore, suppose that~\cref{eq:GrowthBound} holds.
	For every fixed $\varepsilon>0$, there is a set of $\alpha>1$  of full Lebesgue measure such that, for every $\alpha$ in that set and every $x\geq 8$,
	\[
		\abs{ \varSigma_{\alpha}^{(\ell)}(f,x) }
		\ll_{f,\alpha,\varepsilon,\ell} (\log x)^{\frac{1}{2}} (\log\log x)^{\frac{3}{2}+ \varepsilon} \norm{f\rvert_x}_2.
	\]
\end{lem}
\begin{proof}
	Let $\varepsilon>0$ be fixed and consider the sets
	\[
		\mathscr{L}_r
		= \set[\Big]{
			0\leq\lambda<1 :
			\max_{1\leq x\leq 2^r} \abs{\varSigma_{\lambda^{-1}}^{(\ell)}(f,x)}
			> r^{\frac{1}{2}} (\log r)^{\frac{3}{2}+\varepsilon} \norm{f\rvert_{2^r}}_2
		}, \quad r = 3, 4, \ldots;
	\]
	Chebyshev's inequality shows that
	\[
		\meas(\mathscr{L}_r)
		\leq{
			\frac{1}{ r (\log r)^{3+\varepsilon} \norm{f\rvert_{2^r}}_2^2 }
			\int_0^1 \parentheses[\Big]{
				\max_{1\leq x\leq 2^r} \abs{ \varSigma_{\lambda^{-1}}^{(\ell)}(f,x) }
			}^2 \!\dd{\lambda}
		},
	\]
	where $\meas(\,\cdot\,)$ denotes the Lebesgue measure.
	\cref{cor:SigmaL2Bound} (note that $2^r\geq 8$) gives
	\[
		\int_0^1 \parentheses[\bigg]{
			\max_{1\leq x\leq 2^r} \abs{ \varSigma_{\lambda^{-1}}^{(\ell)}(f,x) }
		}^2 \!\dd{\lambda}
		\ll (\log r)^2 \norm{f\rvert_{2^r}}_2^2.
	\]
	Hence, $\meas(\mathscr{L}_r) \ll 1/r(\log r)^{1+2\varepsilon}$, so that
	\[
		\sum_{r=3}^{\infty} \meas(\mathscr{L}_r) < \infty.
	\]
	Then, the Borel--Cantelli lemma yields
	\[
		\meas\parentheses*{ \bigcap_{R=3}^\infty \bigcup_{r=R}^\infty \mathscr{L}_r } = 0.
	\]
	Hence, for every $\epsilon>0$ there is a set $\mathscr{G}_{\epsilon}$ of $\alpha>1$ of full Lebesgue measure such that for every $\alpha$ in that set we have
	\[
		\max_{1\leq x\leq 2^r} \abs{\varSigma_{\lambda^{-1}}^{(\ell)}(f,x)}
		\ll_{f,\alpha,\varepsilon,\ell} r^{\frac{1}{2}} (\log r)^{\frac{3}{2}+\varepsilon} \norm{f\rvert_{2^r}}_2.
	\]
	The dependence of the set $\mathscr{G}_{\epsilon}$ on $\epsilon$ may be removed by considering instead the set $\mathscr{G} = \bigcap_{u=1}^\infty \mathscr{G}_{1/u}$, which still has full Lebesgue measure.
	To finish the proof, note that for any $\epsilon>0$, $\alpha\in\mathscr{G}$ and $x\geq 8$ we have
	\[
		\abs{\varSigma_{\lambda^{-1}}^{(\ell)}(f,x)} \ll_{f,\alpha,\varepsilon,\ell} r^{\frac{1}{2}} (\log r)^{\frac{3}{2}+\varepsilon} \norm{f\rvert_{2^r}}_2 \ll_{f,\alpha,\varepsilon,\ell} (\log x)^{\frac{1}{2}} (\log\log x)^{\frac{3}{2}+\varepsilon} \norm{f\rvert_x}_2,
	\]
	where $2^r$ denotes the smallest power of two exceeding $x$ and the last estimate is justified by~\cref{eq:GrowthBound}.
\end{proof}

\begin{proof}[\textup{We now finish the}~proof of \cref{thm:OurVersion}]
	Indeed, the assertion of the theorem now immediately follows from~\cref{eq:fBeattySummatory} in combination with \cref{lem:SigmaBound}.
\end{proof}

\subsection{Proof of \texorpdfstring{\cref{thm:Optimality}}{Theorem\autoref{thm:Optimality}}}
\label{sec:ProofOfOptimality}

As already mentioned, in the proof we employ a probabilistic argument. To be more specific, we try to randomise the error term $\Sigma_{\alpha}^{(0)}(f,x) - \Sigma_{\alpha}^{(1)}(f,x) $ in \eqref{eq:fBeattySummatory} with respect to the function $f$. Then, we show that with positive probability this random error term satisfies the required relation \eqref{lowerbound}. 
\par A key ingredient we use in the proof is the following variant of the \emph{Law of the Iterated Logarithm} due to Kolmogorov, see  \cite[p.~435]{bingham1986variants-of-lil}.

\begin{thm}[Law of the Iterated Logarithm, Kolmogorov]
	\label{thm:LIL}
	Let $(\Omega, \Sigma, \Prob)$ be a probability space, $(X_m)_{m=1}^{\infty}$ be a sequence of independent random variables, not necessarily identically distributed, and let $S_x=\sum\limits_{m\leq x}X_m$.
	Assume $s_x^2 \coloneqq \Var[S_x] \rightarrow \infty$ and 
	\begin{equation*}
		X_m = o\parentheses*{ \frac{s_m}{\sqrt{\log\log s_m^2}} }, 
		\quad\text{as }m\to\infty \quad \text{almost surely.} 
	\end{equation*} 
	Then 
	\[
		\Prob\parentheses*{ \limsup_{x\to\infty} \frac{S_x}{\sqrt{ 2 s_x^2 \log\log s_x^2 }} = 1 } = 1 .
	\]
\end{thm}

We now consider a probability space $(\Omega, \Sigma, \Prob)$ and  a sequence $(f_m)_{m=1}^{\infty}$ of independent random variables $f_m\colon \Omega\rightarrow \set{\pm1}$ such that
\[
	\Prob( f_m = +1) = \Prob(f_m= -1) = \tfrac{1}{2} \quad (m=1,2,\ldots).
\]
We wish to apply \cref{thm:LIL}  to the sequence of independent random variables 
\[
	X_{m,\lambda}\colon \Omega \rightarrow \RR, \quad
	X_{m,\lambda}(\omega) =
	f_m(\omega) \parentheses{\psi(m\lambda) -\psi((m+1)\lambda)} .
\]
For each $\lambda\in(0,1)$, the random variable $X_{m,\lambda}$ has mean value $\Expectation[X_{m,\lambda}] = 0$ and variance $\Var[X_{m,\lambda}] = (\psi(m\lambda) - \psi((m+1)\lambda))^2$.
Since the $(X_{m,\lambda})_{m=1}^{\infty}$ are independent, we find that for all $x\geq 1$, the random variable
\[
	S_{x,\lambda}\colon\Omega\to\RR,\quad \hspace{3mm}
	S_{x,\lambda}(\omega)  = \sum_{1\leq m\leq x} X_{m,\lambda}(\omega)
\]
has variance
\begin{align*}
	s_x^2(\lambda) &
	\coloneqq \Var[S_{x,\lambda}]
	= \Var\brackets[\Big]{ \sum_{1\leq m\leq x} X_{m,\lambda} }
	= \sum_{1\leq m\leq x} \Var[X_{m,\lambda}] \\
	&= \sum_{1\leq m\leq x} (\psi(m\lambda)- \psi((m+1)\lambda) )^2.
\end{align*}
From now on we restrict ourselves to irrational values of $\lambda\in (0,1)$. Then the sequence $(m\lambda)_{m=1}^\infty$ is uniformly distributed modulo $1$, and therefore
\begin{equation}\label{eq:sx2:Asymptotics}
	\lim_{x\to\infty} \frac{s_x^2(\lambda)}{x}
	= {\int_0^1 (\psi(t)-\psi(t+\lambda))^2 \dd{t}}
	= {\int_0^1 (\boldsymbol{1}_{\lambda^{-1}}(t) - \lambda)^2 \dd{t}}
	= \lambda(1-\lambda).
\end{equation}
(For this and more background on uniform distribution theory, we refer the reader to~\cite{kuipers1974uniformdistribution}.)
In particular $s_x^2(\lambda) \to \infty$ as $x\to\infty$.
\cref{thm:LIL} now yields that for all $\lambda\in (0,1)\setminus \QQ$, we have 
\[
	\limsup_{x\to\infty} \frac{ S_{x,\lambda}(\omega) }{ \sqrt{2s_x^2(\lambda)\log\log s_x^2(\lambda)} } = 1 \quad \text{for $\Prob$-almost all~} \omega\in\Omega.
\]
Employing Fubini's theorem, we can exchange the order of $\lambda$ and $\omega$; that is, we obtain the statement that for $\Prob$-almost all $\omega\in\Omega$
\begin{equation}\label{eq:LIL:Applied2}
	\limsup_{x\to\infty} \frac{ S_{x,\lambda}(\omega) }{ \sqrt{2s_x^2(\lambda)\log\log s_x^2(\lambda)} } = 1 \quad \text{for Lebesgue-almost all~} \lambda\in (0,1).
\end{equation}
Pick one such $\omega_0\in\Omega$ and let $\exceptionalF\colon\NN\to\set{\pm1}$ be the function given by $\exceptionalF(m) = f_m(\omega_0)$.
Then, writing $\lambda = \alpha^{-1}$, equation~\cref{eq:fBeattySummatory} gives
\[
	S_\alpha(\exceptionalF,x) - \alpha^{-1} S(\exceptionalF,x)
	=
	S_{x,\lambda}(\omega)
	+ \LandauO(\alpha)
\]
for every $\alpha>1$. Now combining~\cref{eq:LIL:Applied2} with~\cref{eq:sx2:Asymptotics}, we obtain the conclusion of \cref{thm:Optimality}.

\section{Comments}
\label{sec:Comments}

We close with a short comment on the scope of the method.
Given $\alpha\geq 1$ and a non-negative real number $\beta$ it is also customary to consider the so-called \emph{inhomogeneous} Beatty set $\B(\alpha,\beta) = \set{ \floor{ n\alpha+\beta } : n\in\NN }$.
Consequently, given some arithmetic functions $f\colon\NN\to\CC$, one may wish to compare
\begin{equation}\label{eq:Sfx:Inhom:Def}
	S_{\alpha,\beta}(f,x) =
	\sum_{\substack{ 1\leq m\leq x \\ \mathclap{ m\in \B(\alpha,\beta) } }} f(m)
	\quad\text{with}\quad
	\sum_{1\leq m\leq x} f(m).
\end{equation}

\begin{question}\label{question:InhomogeneousResult}
	Can a corresponding version of \cref{thm:OurVersion} be obtained for inhomogeneous Beatty sets?
\end{question}

In fact, many of the previously cited references concerned with particular choices of~$f$ also treat this more general setting, at essentially no extra effort.
However, it seems that the present method (or the method from~\cite{abercrombie2009arithm} for that matter) does not readily generalise to this modified setup, at least when~$\beta$ is arbitrary.

Nonetheless, when $\beta\geq 0$ is an integer, the argument below works with only minor changes.
To see this, just observe that for inhomogeneous Beatty sets the characterisation in \cref{lem:BeattyMembership} takes the form
\[
	m \in \B(\alpha,\beta)
	\;\Longleftrightarrow\;
	(m-\beta)/\alpha \in \lopeninterval*{ -1/\alpha, 0 } \bmod 1
\]
for any integer $m>\alpha+\beta-1$.
Hence,
\[
	S_{\alpha,\beta}(f,x) = \sum_{\substack{ 1\leq m\leq x \\ \mathclap{ m > \alpha+\beta-1 } }} f(m) \indic_{\alpha}((m-\beta)/\alpha).
\]
From this one easily deduces that
\[
	S_{\alpha,\beta}(f,x)
	= \alpha^{-1} S(f,x)
	+ \varSigma_{\alpha}^{(-\beta)}(f,x) - \varSigma_{\alpha}^{(1-\beta)}(f,x)
	+ \LandauO\parentheses{ (\alpha+\beta) M(f,\alpha+\beta) }.
\]
(See~\cref{eq:fBeattySummatory}.)
Using the fact that $\beta$ is an integer, we obtain the following result as a partial answer towards \cref{question:InhomogeneousResult}.
\begin{thm}
	\label{thm:OurVersion:Inhomogeneous}
	Fix $\varepsilon>0$, and let $\beta$ be a non-negative integer.
	Let $f\colon\NN\to\CC$ be any arithmetic function satisfying~\cref{eq:GrowthBound}.
	Keep the notation from~\cref{eq:Mfx:Def} and~\cref{eq:Sfx:Inhom:Def}.
	Then, for almost all $\alpha>1$ with respect to the Lebesgue measure (depending on $\beta$ and $f$), one has
	\[
		S_{\alpha,\beta}(f,x)
		= \alpha^{-1} S(f,x)
		+ \LandauO_{f,\alpha,\beta,\varepsilon}\parentheses*{
			\norm{f\rvert_{x}}_2
			(\log x)^{\frac{1}{2}}
			(\log\log x)^{\frac{3}{2}+\varepsilon}
		},
	\]
	uniformly for $x\geq 8$.
\end{thm}

It should be noted here that for non-negative, non-integral~$\beta$ our approach fails, because we lack a suitable analogue of \cref{cor:SigmaL2Bound}.
Moreover, when arguing similarly as in~\cite{abercrombie2009arithm}, one faces problems when trying to calculate the $L^2$~norm of the corresponding analogue of~\cref{eq:Qfx:def}, because the shift $\beta\notin\ZZ$ breaks orthogonality of the exponentials which show up after making the necessary adjustments.

\section*{Acknowledgements}
The authors would like to thank Christoph~Aistleitner for insightful conversations leading to the proof of \cref{thm:Optimality} and the TU~Graz for providing excellent working conditions.


\end{document}